\documentclass[10pt]{amsart} 
\usepackage{amsmath,amssymb} 
\newtheorem{theorem}{Theorem}[section]

\newtheorem{lemma}{Lemma}[section] 
\theoremstyle{definition} 
\newtheorem{definition}{Definition}[section]

\usepackage{latexsym,amssymb} 
\begin{document} 

\baselineskip=14.5pt 
\title{A generalization of T\'oth identity involving a Dirichlet Character in $\mathbb{F}_q[T]$}
\author{Esrafil Ali Molla and Subha Sarkar}
\address[Esrafil Ali Molla]{Department of Mathematics, Ramakrishna Mission Vivekananda Educational and Research Institute, G. T. Road, PO Belur Math, Howrah, West Bengal 711202, India}
\address[Subha Sarkar]{Department of Mathematics, National Institute of Technology Jamshedpur, Jharkhand, 831014, India}
\email[Esrafil Ali Molla]{esrafil.math@gmail.com}
\email[Subha Sarkar]{subhasarkar2051993@gmail.com}


\maketitle

\begin{abstract}
Let $\mathbb{A}=\mathbb{F}_q[T]$ be the polynomial ring over the finite field $\mathbb{F}_q$. In this article, we prove a generalization of T\'oth identity to $\mathbb{A}$ involving arithmetical functions, multiplicative and additive characters. 
\end{abstract}

\section{Introduction}

For a positive integer $n$, the classical Menon's identity \cite{menon} states that
\begin{equation}\label{menon}
\displaystyle\sum_{\substack{a=1 \\ (a,n)=1}}^n (a-1,n) =\varphi(n)\sigma_0(n),
\end{equation}
where $\varphi(n)$ is the Euler's totient function and $\sigma_0(n) =\displaystyle \sum_{d \mid n}1$. 

Here and throughout this article, for any $a_1,\dots, a_k\in \mathbb{Z}$, $(a_1,\dots, a_k)$ will always mean $\gcd(a_1,\dots, a_k)$.

\smallskip 

Sury \cite{sury} generalized \eqref{menon} and proved that, for every integers $n \geq 1$ and $s\geq 0$,
\begin{equation}\label{sury}
\sum_{\substack{1 \leq a, b_1, b_2, \ldots, b_s \leq n \\ (a,n)=1}} (a-1,b_1,\ldots, b_s,n)= \varphi(n)\sigma_s(n), \; \text{ where }\; \sigma_s(n) =\displaystyle \sum_{d \mid n} d^s.
\end{equation}

\medspace

Miguel  \cite{m, m1} extended identities \eqref{menon} and \eqref{sury}  from $\mathbb{Z}$ to any residually finite Dedekind domain $\mathcal{D}$.

In another direction, introducing Dirichlet character, Cao and Zhao \cite{zc} proved that 
\begin{equation}\label{zc}
\sum_{\substack{a=1 \\ (a,n)=1}}^n (a-1,n)\chi(a) = \varphi(n) \sigma_0\left(\frac{n}{d}\right),
\end{equation}
where $\chi$ is a Dirichlet character modulo $n$ with conductor $d$. 

\medspace 

Li, Hu and Kim \cite{lhk2} generalized \eqref{sury} and \eqref{zc} to prove that
\begin{equation}\label{lhk}
	\sum_{\substack{1 \leq a, b_1, b_2, \ldots, b_s \leq n \\ (a,n)=1}} (a-1,b_1,\ldots, b_s,n)\chi(a) \lambda_1(b_1)\cdots \lambda_s(b_s)= \varphi(n)\sigma_s\left((\frac{n}{d},w_1,\ldots, w_s)\right),
\end{equation}
where $\chi$ is a Dirichlet character modulo $n$ with conductor $d$, and $\lambda_k(b)=\exp(2 \pi i w_kb/n)$ are additive characters for each $1 \leq k \leq s$.

Later T\'oth \cite{toth3} obtained the most general result in this direction by proving the following identity:

Let $F$ be an arbitrary arithmetical function and $s_j \in \mathbb{Z}$. For each $1 \leq j \leq r$, let $\chi_j$ be Dirichlet characters modulo $n$ with conductor $d_j$, and for each $1 \leq \ell \leq s$, let $\lambda_{\ell}$ be additive characters defined as $\lambda_{\ell}(b)=\exp(2 \pi i w_{\ell}b/n)$, with $w_{\ell} \in \mathbb{Z}$. Then

\begin{eqnarray}\label{tothgen}
	\sum_{\substack{1 \leq a_1,\ldots a_r, b_1, b_2, \ldots, b_s \leq n \\ }} && F((a_1-s_1, \ldots, a_r-s_r,b_1,\ldots, b_s,n)) \chi_1(a_1) \cdots \chi_r(a_r) \lambda_1(b_1)\cdots \lambda_s(b_s) \nonumber \\
	&=& \varphi(n)^r \chi_1^*(s_1) \cdots \chi_r^*(s_r) \sum_{\substack{e\mid (n/d_1,\ldots n/d_r,w_1,\ldots, w_s) \\ (n/e,s_1\cdots s_r)=1 }} \frac{e^s(\mu \ast F)(n/e)}{\varphi(n/e)^r},
\end{eqnarray}
where $\chi_j^*$ are the primitive characters modulo $d_j$ that induce $\chi_j$.

\bigskip

Let $\mathbb{A}=\mathbb{F}_q[T]$ be the polynomial ring over the finite field $\mathbb{F}_q$ with $q$ elements, where $q=p^n$, with $p$ a prime. Recently, Chen and Qi \cite{cq} generalized identity \eqref{tothgen} to $\mathbb{A}$. Before we state the result of \cite{cq}, we introduce the concept of additive character modulo $H$ on $\mathbb{A}$, for any $H \in \mathbb{A}$.

We assume that $H \in \mathbb{A}$ is a fixed polynomial of degree $m$ with the absolute value $|H|=q^{\deg H}$.

\medspace

\begin{definition}
	A non zero function $f: \mathbb{A} \to \mathbb{C} $ is said to be an arithmetical function on $\mathbb{A}$ if
	$$ f(aB)= af(B)\;\;\;\;\;\; \text{ holds for all } a\in \mathbb{F}_q^* \text{ and } B \in  \mathbb{A}.$$
\end{definition}

\medspace

\begin{definition}
Let $f,g: \mathbb{A} \to \mathbb{C}$ be two arithmetic functions on $\mathbb{A}$. The Dirichlet convolution $f \ast g$ is defined by
$$f\ast g (H)=\sum_{D \mid H} f(D)g\left(\frac{H}{D}\right).$$
\end{definition}

\smallskip

\begin{definition}
Let $ H \in \mathbb{A}.$ Then we define the following functions on $\mathbb{A}$. 
\begin{enumerate}
		\item[(i)] The M\"obius $\mu$ function is defined as 
		\begin{equation*}
			\mu(H) :=
			\begin{cases}
				1   & ~ \text{ if } H \in \mathbb{F}_q^*,\\
				(-1)^{t} &  ~ \text{ if } H = \displaystyle\prod_{\i = 1}^{t}P_{i}; \mbox{ where each } P_{i}  \mbox{ is a monic irreducible polynomial and } P_{i} \neq P_{j} \mbox{ for } i \neq j,\\
				0 & ~ \text{ otherwise. }
			\end{cases}
\end{equation*}
		
\medskip
		
\item[(ii)] The Euler function $\varphi(H)$ is defined as $$\varphi(H) := |H| \prod_{P\mid H}\left(1-\frac{1}{|P|}\right),$$ 
		which counts the number of polynomials of degree less than $\deg H$ that are coprime to $H$. 
		
	\end{enumerate}
\end{definition}

\bigskip

\begin{definition}
	Let $H \in \mathbb{A}$. A Dirichlet character $\chi$ modulo $H$ is a function $\chi:\mathbb{A} \rightarrow \mathbb{C}$ that satisfies
	
	\subitem(i) $\chi(A+KH)=\chi(A)$ for all $A,K \in \mathbb{A}$.
	
	\subitem(ii) $\chi(A_1A_2)= \chi(A_1)\chi(A_2)$ for all $A_1,A_2 \in \mathbb{A}$.
	
	\subitem(iii) $\chi(A) \neq 0$ if and only if $(A,H)=1$.
\end{definition}

\smallskip	

	\begin{definition}
		Let $\chi$ be a Dirichlet character modulo $H$.
	\begin{enumerate}
		\item[(i)] A monic divisor $D$ of $H$ is said to be an {\it induced modulus} for $\chi$ if $\chi(A)=\chi(B)$ whenever $A,B \in \mathbb{A}$ and $A \equiv B \pmod{D}$. The unique induced modulus $D_0$ such that for any induced modulus $D$, we have $D_0 \mid D$, is said to be the {\it conductor} of $\chi$.
		
		\medspace
		
		\item[(ii)]A character $\chi$ modulo $H$ is said to be a {\it primitive character modulo} $H$ if it has no induced modulus $D$ with $D \neq H$. If $\chi$ is any character modulo $H$ with conductor $D$, then there is a unique primitive character $\psi$ modulo $D$ such that  $\chi(A)=\psi(A)$ for all $A \in \mathbb{A}$.
	\end{enumerate}
\end{definition}

\medspace

\begin{definition}
	Let $H \in \mathbb{A}$ be a polynomial of degree $m$. Then for any $B \in \mathbb{A}$ with $B\equiv b_{m-1}T^{m-1} + \cdots + b_1T + b_0 \pmod{H}$, we define an additive function $t$ modulo $H$ by
	$$ t(B) = b_{m-1}.$$ 
	
	Again, for any $G\in \mathbb{A}$, if we define $t_G(B)= t(GB)$, then $t_G(B)$ is also an additive function modulo $H$. Now we define an additive character modulo $H$ on $\mathbb{A}$ by
	$$ E(G,H)(B)= \exp{ \left(\frac{2\pi i}{p} Tr(t_G(B)) \right)} \text{ for all } B \in \mathbb{A},$$
	where $Tr: \mathbb{F}_q \to \mathbb{F}_p$ is the trace map.
\end{definition}

The following lemma says that any additive character $\lambda$ modulo $H$ corresponds to some $W \in \mathbb{A}$.

\begin{lemma} \cite{cq} \label{add}
	For any additive character $\lambda$ modulo $H$ on $\mathbb{A}$, there exists a unique $W \in \mathbb{A}$ such that $ \lambda= E(W,H)$ and $\deg(W) < \deg(H)$.
\end{lemma}

\medspace

The result in \cite{cq} is the following extension of T\'oth's identity \eqref{tothgen} in $\mathbb{A}$:

Let $f$ be any arithmetical function on $\mathbb{A}$ and $H, K_j, B_i, S_j \in \mathbb{A}$. For each $1 \leq j \leq r$, let $\chi_j$ be Dirichlet characters modulo $H$ with conductors $D_j$, and for each $1 \leq i \leq s$, let $\lambda_i$ be additive characters with $W_i \in \mathbb{A}$ corresponds to $\lambda_i$ as in Lemma \ref{add}. Then

\begin{eqnarray}\label{cq}
	\sum_{\substack{ K_1,\ldots K_r, B_1, B_2, \ldots, B_s \pmod H \\ }} && f((K_1-S_1, \ldots, K_r-S_r,B_1,\ldots, B_s,H)) \chi_1(K_1) \cdots \chi_r(K_r) \lambda_1(B_1)\cdots \lambda_s(B_s) \nonumber \\
	&=& \varphi(H)^r \chi_1^*(S_1) \cdots \chi_r^*(S_r) \sum_{\substack{M\mid (\frac{H}{D_1},\ldots \frac{H}{D_r},W_1,\ldots, W_s) \\ (\frac{H}{M},S_1\cdots S_r)=1 }} \frac{|M|^s(\mu \ast f)(\frac{H}{M})}{\varphi(\frac{H}{M})^r},
\end{eqnarray}
where $\chi_j^*$ are the primitive characters modulo $D_j$ that induce $\chi_j$ for each $1 \leq j \leq r$.

Here and throughout this article, for any $A_1,\dots, A_k \in \mathbb{A}$, $(A_1,\dots, A_k)$ will denote the $\gcd(A_1,\dots, A_k)$.

\bigskip

For every positive integer $k$, T\'{o}th \cite{toth} introduced the $k$-dimensional generalized Euler's totient function $\varphi_{k}$ as,
\begin{equation*}
	\varphi_k(n)= \sum_{\substack{a_1, a_2, \ldots, a_k =1 \\ (a_1\cdots a_k,n)=1 \\ (a_1+\cdots+a_k,n)=1}}^n 1.
\end{equation*} 

For every positive integer $k$, he proved that the function  $\varphi_k$ is multiplicative and $$\varphi_k(n)= \varphi(n)^k \prod_{p \mid n} \left(1-\frac{1}{p-1}+\frac{1}{(p-1)^2}+\cdots+ \frac{(-1)^{k-1}}{(p-1)^{k-1}}\right).$$
He also proved that the following Menon-type identity
\begin{equation}\label{toth}
	\sum_{\substack{a_1, a_2, \ldots, a_k =1 \\ (a_1\cdots a_k,n)=1 \\ (a_1+\cdots+a_k,n)=1}}^n (a_1+\cdots+a_k-1,n) = \varphi_k(n)\sigma_0(n)\;\;\;\;\;\;\;\; \text{ holds. }
\end{equation}

In \cite{js,subha}, T\'oth's identity \eqref{toth} has been generalized to the ring of algebraic integers $\mathcal{O}_K$ concerning arithmetical functions and Dirichlet characters. In the case of $\mathcal{O}_K=\mathbb{Z}$, the identity of \cite{subha} reads as follows:

Let $n$ be a positive integer. Let $\chi$ be a Dirichlet character modulo $n$ with conductor $d$. Then for a fixed positive integer $r$ with $(r,n)=1$ and for any arithmetical function $F$, we have

\begin{multline}\label{meq 8}
	\sum_{\substack{1\leq  a_1, a_2,\ldots, a_k,b_1, \ldots, b_s \leq n \\ (a_1+a_2+\cdots+a_k,n)=1 \\ (a_1a_2\cdots a_k,n)=1}} F((a_1+a_2+\cdots+a_k-r,b_1,b_2,\ldots, b_s,n))\chi(a_1)\\ = \mu(d)^{k-1}\psi(r)\varphi\left(\frac{n_{0}^{k}}{d^{k-1}}\right)\varphi_k\left(\frac{n}{n_{0}}\right) \sum_{\substack{d\mid e\mid n\\ e\mid b_1,\ldots,e\mid b_s}}\frac{(\mu \ast F)(e)}{\varphi(e)},
\end{multline}
where $\psi$ is the primitive character modulo $d$ that induces $\chi$, $n_{0} \mid n$ is such that $n_0$ has the same prime factors as that of $d$ and $\left(n_0, \frac{n}{n_0}\right) = 1$.

\bigskip

In this article, we give an extension of identity \eqref{meq 8} to $\mathbb{A}$. But before that, we need to define the $k$-dimensional generalized Euler function $\varphi$ in $\mathbb{A}$.

\medspace

\begin{definition}
For any positive integer $k$, we define the $k$-dimensional generalized Euler function $\varphi_k$ on $\mathbb{A}$ as,
$$ \varphi_k(H)= \sum_{\substack{H_1, H_2, \ldots, H_k \pmod{H}\\(H_1 \cdots H_k,H)=1 \\ (H_1+\cdots+H_k,H)=1}} 1. $$
\end{definition} 
Note that $\varphi _1(H) =\varphi(H)$.

\medspace

The main results of this article are as follows. 
\begin{theorem}\label{MAIN_TH1} 
Let $H \in \mathbb{A}$. Then for every positive integer $k$, we have  $$\varphi_k(H)= \varphi(H)^k \prod_{P \mid H} \left(1-\frac{1}{|P|-1}+\frac{1}{(|P|-1)^2}-\cdots+ (-1)^{k-1}\frac{1}{(|P|-1)^{k-1}}\right).$$ 
\end{theorem} 

\medskip 

\begin{theorem}\label{MAIN-TH2} 
Let $H, K_j, B_i\in \mathbb{A}$. Let $\chi$ be a Dirichlet character modulo $H$ with conductor $D$ and for each $1\leq i \leq s$, $\lambda_i$ additive characters with $W_i \in \mathbb{A}$ corresponds to $\lambda_i$, as in Lemma \ref{add}. Then for a fixed element $S \in \mathbb{A}$ with $(S,H)=1$ and for any arithmetical function $F$ on $\mathbb{A}$, we have 
\begin{multline}\label{meq}
    \sum_{\substack{ K_1, K_2,\ldots, K_l, B_1,B_2,\ldots, B_s \pmod{H}\\(K_1+K_2+\cdots+K_l, H)=1\\ (K_1K_2\cdots K_l, H)=1}} F((K_1+K_2+\cdots+K_l-S,B_1,B_2,\ldots, B_s,H))\chi(K_1) \lambda_1(B_1) \cdots \lambda_s(B_s)\\ =  \mu(D)^{l-1} \; \psi(S) \; \varphi\left( \frac{H_0^l}{D^{l-1}}\right) \varphi_l \left( \frac{H}{H_0}\right) \sum_{\substack{D\mid M \mid H \\ \frac{H}{M} \mid W_1, W_2, \ldots, W_s}}\frac{ (\mu \ast F)(M)  }{ \varphi(M)} \left|\frac{H}{M}\right|^s,
\end{multline}
where $\psi$ is the primitive character modulo $D$ that induces $\chi$, $H_0 \mid H$ is such that $H_0$ has the same prime factors as that of $D$ and $\left(H_0, \frac{H}{H_0} \right)=1$.
\end{theorem}

\section{Preliminaries}
Let $p$ be a fixed prime, and $q=p^n$. Let $\mathbb{A}=\mathbb{F}_q[T]$ be the polynomial ring over the finite field $\mathbb{F}_q$. We will prove some lemmas which will be useful to prove our main theorems.

\begin{lemma}\label{lem1}
Let $H, D$ be two polynomials in $\mathbb{A}$ such that $D\mid H.$ Then for any $Q \in \mathbb{A}$ with $(Q,H)=1,$ we have
\begin{equation*}
 \vert \{ K \pmod{H} : (K, H)=1 \text{ and } K \equiv Q \pmod D \} \vert = \frac{\varphi(H)}{\varphi(D)}.
\end{equation*}
\end{lemma}

\begin{proof}
We define a map $ w : (\mathbb{A}/ H)^\times \rightarrow  (\mathbb{A}/ D)^\times $ by 
$$
w(X\pmod{H}) = X \pmod{D}. 
$$
It is clear that $w$ is a group homomorphism. Now we want to show that $w$ is an epimorphism.

Let $H= P_1^{\alpha_1} \cdots P_t^{\alpha_t} Q_1^{\gamma_1} \cdots Q_l^{\gamma_l}$ and $ D = P_1^{\beta_1} \cdots P_t^{\beta_t} $ with $ \beta_i \le \alpha_i.$
Then for any $Y \in \mathbb{A}$ such that $(Y,D)=1,$ the system of congruences
\begin{equation*}
\begin{split}
& X\equiv Y \pmod{P_i^{\alpha_i} } \quad  \forall \quad1 \le i \le t \\
& X\equiv 1 \pmod{Q_j^{\gamma_j}} \quad  \forall \quad1 \le j \le l
\end{split}
\end{equation*}
has a solution by Chinese Remainder Theorem. Hence $(X,H)=1 $ and $w(X)=Y$. This shows that $w$ is an epimorphism.

We have
\begin{equation} \label{2.1 eq1}
\begin{split}
\ker (w) &= \left\{ K \pmod H \in (\mathbb{A}/H)^\times: w(K)=1\right\} \\
&=  \left\{ K \pmod H : (K,H)=1 \text{ and } K \equiv 1 \pmod D\right\} .   
\end{split}
\end{equation}

Also for any $Q\in \mathbb{A} $ with $ (Q,H)=1$, we show that the function 
$$f: \left\{ K \pmod H :  (K,H)=1 \text{ and } K \equiv 1 \pmod D\right\} \rightarrow \left\{ K \pmod{H} : (K, H)=1 \text{ and } K \equiv Q \pmod D \right\}$$
given by 
$$
f(P) = Q P \pmod H $$
is a bijection.

$f$ is one-to-one, as if $f(P_1)= f(P_2) \implies QP_1=QP_2 \pmod H$.
Multiplying both sides by the inverse $Q^{-1}$ of $Q$ in $(\mathbb{A}/ H)^\times$, we get $P_1= P_2 \pmod H$.

Again for any $K \in (\mathbb{A}/ H)^\times $ such that $K \equiv Q \pmod D$, if we take $P = Q^{-1}K \pmod H $, then clearly $(Q^{-1}K,H)=1$, $ Q^{-1}K \equiv 1 \pmod D$ and $ f(Q^{-1}K) = Q Q^{-1}K \pmod H = K \pmod H $. This proves $f$ is a bijection.

Therefore 
\begin{equation} \label{2.1 eq2}
\begin{split}
& \left|\{ K \pmod H :  (K,H)=1 \text{ and } K \equiv 1 \pmod D \}\right| \\
& \quad \quad \quad=  \left| \{ K \pmod{H} : (K, H)=1 \text{ and } K \equiv Q \pmod D  \}\right|.
\end{split}
\end{equation}

Combining \eqref{2.1 eq1} and \eqref{2.1 eq2} we get
\begin{equation*}
  |\ker (w)| = \left| \{ K \pmod{H} : (K, H)=1 \text{ and } K \equiv Q \pmod D \} \right|.
\end{equation*}

Therefore by the first group isomorphism theorem, we have
$$  
(\mathbb{A}/ H)^\times / \ker (w) \cong (\mathbb{A}/ D)^\times.
$$
It follows that
\begin{equation*}
    \vert   \ker (w) \vert = \left|  \frac{(\mathbb{A}/ H)^\times}{(\mathbb{A}/ D)^\times}  \right| .
\end{equation*}
This implies
\begin{equation*}
 \vert \{ K \pmod{H} : (K, H)=1 \text{ and } K \equiv Q \pmod D \} \vert = \frac{\varphi(H)}{\varphi(D)}.
\end{equation*}
\end{proof}

\medspace

Let $H, D\text{ and } M \in \mathbb{A}$ such that $ D, M \mid H$. Then for any $Q, S \in \mathbb{A}$, we define
 
\begin{equation}\label{SH}
 S_H(D,M, Q, S):= \{ K \pmod H : (K,H)=1 , K\equiv Q \pmod D, K \equiv S \pmod M \}
\end{equation}
and
\begin{equation}\label{gH}
g_H(D,M , Q, S) := \vert S_H(D,M, Q, S) \vert.
\end{equation}

\medspace

\begin{lemma}\label{L2}
Let $H= P_1^{\alpha_1} \cdots P_t^{\alpha_t}$ where $P_1, \dots, P_t$ be distinct irreducible polynomials and $\alpha_i \ge 1$ integers. Let $ D= P_1^{\beta_1} \cdots P_t^{\beta_t}$, $M = P_1^{\gamma_1} \cdots P_t^{\gamma_t}$ be such that $0 \le \beta_i, \gamma_i \le \alpha_i$. Then for any $Q,S \in \mathbb{A}$, we have
$$ 
g_H(D,M,Q,S) = \prod_{i=1}^{t} g_{P_i^{\alpha_i}}(P_i^{\beta_i}, P_i^{\gamma_i}, Q,S).
$$
\end{lemma}

\begin{proof}
It is sufficient to show that 
$$
g_H(D,M,Q,S) =g_{P_1^{\alpha_1}}(P_1^{\beta_1}, P_1^{\gamma_1}, Q,S) g_{H^{'}}(D',M',Q,S),
$$
where $ H= P_1^{\alpha_1} H^{'}, \,  D=  P_1^{\beta_1}D^{'} \text{ and } M= P_1^{\gamma_1} M^{'}$, and then use the induction on $t$.

\medskip

Define a map $\phi: S_H(D,M,Q,S)  \rightarrow   S_{P_1^{\alpha_1}}(P_1^{\beta_1}, P_1^{\gamma_1}, Q,S) \times  S_H^{'}(D^{'},M^{'},Q,S)$ by
$$
\phi(A) = (A_1,A_2),
$$
where $A_1\in (\mathbb{A}/P_1^{\alpha_1})^\times$, $ A_2 \in (\mathbb{A}/ H')^\times$ be such that $A_1 \equiv A \pmod{ P_1^{\alpha_1}}$, $A_2 = A \pmod{H'}$.

\medskip

It is easy to check that $A_1 \in  S_{P_1^{\alpha_1}}(P_1^{\beta_1}, P_1^{\gamma_1}, Q,S) , A_2 \in  S_H'(D',M',Q,S)$ and $\phi$ is well defined.

To prove $\phi$ is injective, let $A, A^{'} \in S_H(D,M,Q,S)$ be such that $\phi(A) =(A_1, A_2)= (A_1^{'}, A_2^{'})= \phi(A^{'}).$ 
Then $ A\equiv A^{'}\pmod{ P_1^{\alpha_1}} $ and $ A\equiv A^{'} \pmod{H^{'}}$. As $(H^{'}, P_1^{\alpha_1}) =1,$ we have $A\equiv A^{'} \pmod{H}$.

To prove $\phi $ is surjective, let $(A_1, A_2) \in  S_{P_1^{\alpha_1}}(P_1^{\beta_1}, P_1^{\gamma_1}, Q,S) \times  S_H'(D^{'},M^{'},Q,S)$. Since $(P_1^{\alpha_1}, H^{'}) =1,$
by Chinese Remainder Theorem, there exists $A\in (\mathbb{A}/ H)^\times$ such that $ A\equiv A_1\pmod{ P_1^{\alpha_1}} $ and $ A\equiv A_2 \pmod{ H^{'}}$. 
\end{proof}

For simplicity, we write $[A_1,\dots, A_k ] = \text{lcm}(A_1,\dots, A_k)$ for any $A_1,\dots, A_k \in \mathbb{A}$, throughout the following.

\begin{lemma}\label{lem2}
Let $H,D, M, Q, S \in \mathbb{A}$, such that $D, M\mid H$. Then
\begin{equation*}
\sum_{\substack{ K \pmod{H} \\ (K,H)=1\\ K \equiv Q \pmod{D}\\K \equiv S \pmod{M}}}1= 
\begin{cases} 
\frac{\varphi( H)}{\varphi([D,M])} & \text{ if } (Q, D)=(S,M)=1 \text{ and } Q\equiv S \pmod{(D,M)} \\ 
0 & \text{ otherwise, } 
\end{cases} 
\end{equation*} 
\end{lemma}

\begin{proof}
From (\ref{SH}) and (\ref{gH}), we get 
\begin{equation*}
\sum_{\substack{ K \pmod{H} \\ (K,H)=1\\ K \equiv Q \pmod{D}\\K \equiv S \pmod{M}}}1 = g_H(D,M , Q, S) .
\end{equation*}

By Lemma \ref{L2}, it is enough to prove this lemma when $ H= P^\alpha, \, D= P^\beta, \, M= P^\gamma$ with $ 0\le \beta , \gamma \le \alpha .$

As $(K, H)= 1$, $D\mid H$ and $ M\mid H, $  it is obvious that $(K,D)= (K,M)=1.$ Since $K \equiv Q \pmod{D}$ and $K \equiv S \pmod{M}$, we have $(Q, D)=(S,M)=1$.
Thus if $(Q, P^\beta)\neq 1$ or $(S, P^\beta) \neq 1,$ then 
$$g_H(P^\beta, P^\gamma,S,M)=0. $$
We consider the case $(Q,P^\beta)=1$ and $(S, P^\gamma)=1$.
If $Q \not\equiv S \pmod{ (P^\beta, P^\gamma) }$, then also $$ g_{P^\alpha}(P^\beta, P^\gamma, Q,S)=0. $$
Thus we assume that $Q \equiv S \pmod{ (P^\beta, P^\gamma) } \equiv S \pmod{ P^{\, \min \{\beta, \gamma\}}}$ and $ (Q, P^\beta)=(S, P^\gamma)=1$.\\
It is clear that $[P^\beta, P^\gamma]= P^{\,\max\{\beta, \gamma\}}$. By Lemma \ref{lem1}, we get
$$
g_{P^\alpha}(P^\beta, P^\gamma, Q,S)= \frac{\varphi(P^\alpha)}{\varphi([P^\beta, P^\gamma])},
$$
and by Lemma \ref{L2}, this proves the lemma.
\end{proof}

\medskip

\begin{lemma} \label{char}
	Let $D,S \in \mathbb{A}$ be such that $(S,D)=1$ and $\psi$ a primitive character modulo $D$. Then for any $Q\mid D$, we have  
	$$
	\sum_{\substack{ U \pmod{D} \\ U\equiv S \pmod{Q}}} \psi (U) \neq 0  \;\;\;\; \text{  if and only if } \;\;\;\; Q=D.
	$$
	In particular, if $ Q=D ,$ then 
	$$
	\sum_{\substack{ U \pmod{D} \\ U\equiv S \pmod{Q}}} \psi (U) = \psi(S).
	$$
\end{lemma}
\begin{proof}
	{\noindent \sf Case I:} Let $Q= D$. Since $(S,D)=1,$ we have $\psi(S)\neq 0.$
	Hence 
	$$
	\sum_{\substack{ U \pmod{D} \\ U\equiv S \pmod{D}}} \psi (U) = \psi(S).
	$$
{\noindent \sf Case II:} Let $Q=1$. Since $\psi$ is a primitive character modulo $D,$ Then 
	$$
	\sum_{\substack{ U \pmod{D} \\ U\equiv S \pmod{Q}}} \psi (U) = \sum_{\substack{ U \pmod{D} }} \psi (U)=0.
	$$
{\noindent \sf Case III:} Let $Q \neq 1 $, $ Q\neq D$ and 
	$D= P_1^{\alpha_1} \cdots P_m^{\alpha_m}$,
	where $P_1, \dots, P_m$ are irreducible polynomials in $\mathbb{A}$ and $\alpha_1, \dots , \alpha_m$ are positive integers. Since $Q\mid D$, we have $Q= P_1^{\beta_1} \cdots P_m^{\beta_m}$, where $0\le \beta_i \le \alpha_i $ and $i= 1,2 ,\dots , m$.
	
Since $\psi$ is a primitive character modulo $D$, there are primitive characters $\psi_{P_i^{\alpha_i}}$ modulo $ P_i^{\alpha_i}$  for each $i= 1,2,\dots, m$ such that
\begin{equation*}
\begin{split}
\psi(U) &= \psi_{P_1^{\alpha_1}} \cdots  \psi_{P_m^{\alpha_m}}(U) \\
	&= \psi_{P_1^{\alpha_1}}(U) \cdots  \psi_{P_m^{\alpha_m}}(U).\\    
\end{split}
\end{equation*}
By Chinese Remainder Theorem, we have an isomorphism 
$$
\mathbb{A}/ D \cong \mathbb{A}/ {P_1^{\alpha_1}}  \oplus \cdots \oplus \mathbb{A}/{P_m^{\alpha_m}}.
$$
Hence 
\begin{equation}\label{L42}
	\sum_{\substack{ U \pmod{D} \\ U\equiv S \pmod{Q}}} \psi (U) = \prod_{i=1}^{m} \quad \sum_{\substack{ U_i \pmod{ P_i ^{\alpha_i} } \\ U_i \equiv S \pmod{ P_i ^{\beta_i}}}} \psi_{P_i^{\alpha_i}}(U_i).
\end{equation}
Since $Q\neq D$ and $Q\neq 1$, there exists $j \in \{ 1,2, \dots, m \}$ such that $ 0< \beta_j < \alpha_j$. \\
We now show that
	$$
	\sum_{\substack{ U\pmod{ P_j ^{\alpha_j} } \\ U \equiv S \pmod{ P_j ^{\beta_j}}}} \psi_{P_j^{\alpha_j}}(U) =0.
	$$
	Since $(S,D)= 1,$ we have $(S, P_j)= 1$, hence
	
	\begin{equation}\label{L43}
		\begin{split}
			\sum_{\substack{ U \pmod{ P_j ^{\alpha_j} } \\ U \equiv S \pmod{ P_j ^{\beta_j}}}} \psi_{P_j^{\alpha_j}}(U) &= \sum_{\substack{ US \pmod{ P_j ^{\alpha_j} } \\ US \equiv S \pmod{ P_j ^{\beta_j}}}} \psi_{P_j^{\alpha_j}}(US) \\ 
			&= \psi_{P_j^{\alpha_j}}(S)  \sum_{\substack{ U\pmod{ P_j ^{\alpha_j} } \\ U \equiv 1 \pmod{ P_j ^{\beta_j}}}} \psi_{P_j^{\alpha_j}}(U).
		\end{split}
	\end{equation}
	Since $\psi_{P_j^{\alpha_j}} $ is a primitive charater modulo $P_j^{\alpha_j}$ and $0< \beta_j < \alpha_j,$ there exists an element $G \in \mathbb{A}$ such that $G\equiv 1 \pmod{ P_j^{\beta_j}}$ and $\psi_{P_j^{\alpha_j}}(G) \neq 1$. \\
Now
	\begin{equation*}
		\begin{split}
			\psi_{P_j^{\alpha_j}}(G)  \sum_{\substack{ U\pmod{ P_j ^{\alpha_j} } \\ U \equiv 1 \pmod{ P_j ^{\beta_j}}}} \psi_{P_j^{\alpha_j}}(U) &= \sum_{\substack{ U\pmod{ P_j ^{\alpha_j} } \\ U \equiv 1 \pmod{ P_j ^{\beta_j}}}} \psi_{P_j^{\alpha_j}}(UG) \\
			&= \sum_{\substack{ UG^{-1}\pmod{ P_j ^{\alpha_j} } \\ U G^{-1} \equiv 1 \pmod{ P_j ^{\beta_j}}}} \psi_{P_j^{\alpha_j}}(U) \\
			&=  \sum_{\substack{ U\pmod{ P_j ^{\alpha_j} } \\ U \equiv 1 \pmod{ P_j ^{\beta_j}}}} \psi_{P_j^{\alpha_j}}(U) .\\
		\end{split}
	\end{equation*}
	Since $\psi_{P_j^{\alpha_j}}(G) \neq 1$, this implies
	$$
	\sum_{\substack{ U\pmod{ P_j ^{\alpha_j} } \\ U \equiv 1 \pmod{ P_j ^{\beta_j}}}} \psi_{P_j^{\alpha_j}}(U) =0.
	$$
	From \eqref{L43} we get,
	$$
	\sum_{\substack{ U \pmod{ P_j ^{\alpha_j} } \\ U \equiv S \pmod{ P_j ^{\beta_j}}}} \psi_{P_j^{\alpha_j}}(U) =0.
	$$
	Hence from \eqref{L42}, 
	$$
	\sum_{\substack{ U \pmod{D} \\ U\equiv S \pmod{Q}}} \psi (U) =0.
	$$
	
\end{proof}

\section{Proof of Theorem \ref{MAIN_TH1}}

Let $H,M \in \mathbb{A}$ be two polynomials. For every $k \geq 1$, we define a more general function than $\varphi_k(H),$ namely
$$ \varphi_k(H,M):= \sum_{\substack{K_1, K_2, \ldots, K_k  \pmod{H} \\(K_1\cdots K_k,H)=1 \\ (K_1+\cdots+K_k,M)=1}} 1. $$\\
Note that $\varphi_k(H,M)= \varphi_k(H)$ whenever $M=H$. \\

We first prove the following lemma. 
\begin{lemma}\label{LM2}(Recursion formula for $\varphi_k(H,M))$
Let $k\geq 2$ and $M,H \in \mathbb{A}$ such that $M \mid H$. Then 
\begin{equation*}
    \varphi_k(H,M)= \varphi(H) \sum_{\substack{D\mid M }}\frac{\mu(D)}{\varphi(D)} \varphi_{k-1}(H,D) .
\end{equation*}
\end{lemma}

\begin{proof}
	We have
\begin{equation*} 
\begin{split}
     \varphi_k(H,M) =& \sum_{\substack{K_1, K_2, \ldots, K_k  \pmod{H} \\(K_1\cdots K_k,H)=1 \\ (K_1+\cdots+K_k,M)=1}} 1 \\ 
     =& \sum_{\substack{ K_1, K_2, \dots, K_k \pmod{ H} \\ (K_1 K_2 \cdots K_k,H )=1}} \, \sum_{\substack{D\mid  (K_1+ K_2+ \cdots+ K_k,M)}}\mu(D) \\
     =& \sum_{D\mid M}\mu(D) \, \sum_{\substack{ K_1, K_2, \dots, K_k \pmod{ H} \\ (K_1 K_2 \cdots K_k,H)=1\\ K_1+ K_2+ \cdots+ K_k \equiv 0 \pmod{D}}}1 \\
     =& \sum_{D\mid M}\mu(D) \,\sum_{\substack{ K_1, K_2, \dots, K_{k-1} \pmod{ H} \\ (K_1 K_2 \cdots K_{k-1},H)=1}} \sum_{\substack{K_k \pmod{H} \\ (K_k, H)=1 \\ K_k \equiv-(K_1 + K_2 + \dots + K_{k-1}) \pmod D}}1.
    \end{split}
\end{equation*}
By using Lemma \ref{lem1} we get,
     \begin{equation*}
     \begin{split}
     \varphi_k(H,M) =& \sum_{D\mid M}\mu(D) \, \sum_{\substack{ K_1, K_2, \dots, K_{k-1} \pmod{ H} \\ (K_1 K_2 \cdots K_{k-1},H)=1 \\(K_1 + K_2 + \cdots + K_{k-1} ,D)=1 }}\frac{\varphi(H)}{\varphi(D)} \quad \\
     =&\; \varphi(H) \sum_{D\mid M}\frac{\mu(D)}{\varphi(D)} \, \sum_{\substack{ K_1, K_2, \dots, K_{k-1} \pmod{ H} \\ (K_1 K_2 \cdots K_{k-1},H)=1 \\(K_1 + K_2 + \cdots + K_{k-1} ,D)=1 }}1 \\
     =&\; \varphi(H) \sum_{D\mid M}\frac{\mu(D)}{\varphi(D)} \, \varphi_{k-1}(H,D). 
\end{split}
\end{equation*}
\end{proof}

{\noindent \bf Proof of Theorem \ref{MAIN_TH1}.} 
Let $H,M \in \mathbb{A}$ and $M \mid H$. We show by induction on $k$ that

\begin{equation}\label{AB1}
    \varphi_k(H,M)= \varphi(H)^k \prod_{P\mid M}\left(1-\frac{1}{|P|-1}+\frac{1}{(|P|-1)^2}-\cdots+ \frac{(-1)^{k-1}}{(|P|-1)^{k-1}}\right) 
\end{equation}
holds for every $k \in \mathbb{N}$.

If $k=1$, then $\varphi_1(H,M)= \varphi(H)$. Let $k\geq 2$, and assume that \eqref{AB1} holds for $k-1$. Using Lemma \ref{LM2} we get,
\begin{equation*}
\begin{split}
\varphi_{k}(H,M) &= \varphi(H) \; \sum_{\substack{D\mid M }}\frac{\mu(D)}{\varphi(D)} \; \varphi_{k-1}(H,D) \\ 
&= \varphi(H) \sum_{D\mid M} \frac{\mu(D)}{\varphi(D)} \;\; \varphi(H)^{k-1} \; \prod_{P\mid D}\left(1-\frac{1}{|P|-1}+\frac{1}{(|P|-1)^2}-\cdots+ \frac{(-1)^{k-2}}{(|P|-1)^{k-2}}\right) \\
&= \varphi(H)^k \prod_{P\mid M} \left(1 + \frac{\mu(P)}{\varphi(P)} \left(1-\frac{1}{|P|-1}+\frac{1}{(|P|-1)^2}-\cdots+ \frac{(-1)^{k-2}}{(|P|-1)^{k-2}}\right) \right) \\
&= \varphi(H)^k \prod_{P\mid M}\left(1-\frac{1}{|P|-1}+\frac{1}{(|P|-1)^2}-\cdots+ \frac{(-1)^{k-1}}{(|P|-1)^{k-1}}\right).
\end{split}
\end{equation*}

Therefore \eqref{AB1} holds for every $k\in \mathbb{N}$. Now if we take $M=H$, then we have the proof of Theorem \ref{MAIN_TH1}.

\section{Proof of Theorem \ref{MAIN-TH2} }

Let $\mathcal{T}$ denote the sum of the LHS of \eqref{meq}.

\begin{equation*}
\begin{split}
\mathcal{T} &= \sum_{\substack{ K_1, K_2,\ldots, K_l, B_1,B_2,\ldots, B_s \pmod{H}\\(K_1+K_2+\cdots+K_l, H)=1\\ (K_1K_2\cdots K_l, H)=1}} F((K_1+K_2+\cdots+K_l-S,B_1,B_2,\ldots, B_s,H))\chi(K_1) \lambda_1(B_1) \cdots \lambda_s(B_s)\\
&= \sum_{U \pmod{D}}\psi(U)\sum_{\substack{K_1, K_2,\ldots, K_l, B_1,B_2,\ldots, B_s \pmod{H}\\(K_1+K_2+\cdots+K_l, H)=1\\ (K_1K_2\cdots K_l, H)=1 \\K_1\equiv U \pmod{D}}} F((K_1+\cdots+K_l-S,B_1,\ldots, B_s,H))\lambda_1(B_1) \cdots \lambda_s(B_s) .
\end{split}
\end{equation*}
By using the convolution identity $F(H)=\displaystyle\sum_{M \mid H}(\mu \ast F)(M)$, we get

\begin{equation}\label{sum}
\begin{split}
   \mathcal{T} &=\sum_{U \pmod{D}}\psi(U)\sum_{\substack{K_1, K_2,\ldots, K_l, B_1,B_2,\ldots, B_s \pmod{H}\\(K_1+K_2+\cdots+K_l, H)=1\\ (K_1K_2\cdots K_l, H)=1 \\K_1\equiv U \pmod{D}}} \lambda_1(B_1) \cdots \lambda_s(B_s) \sum_{M \mid (K_1+K_2+\cdots+K_l-S,B_1,B_2,\ldots, B_s,H)} (\mu \ast F)(M)\\
&= \sum_{\substack{M \mid H}}(\mu \ast F)(M)\sum_{\substack{B_1,B_2,\ldots, B_s \pmod{H}\\ M \mid B_1,\ldots, M \mid B_s}} \lambda_1(B_1) \cdots \lambda_s(B_s)\sum_{U \pmod{D}}\psi(U) \sum_{\substack{K_1, K_2,\ldots, K_l\pmod{H}\\(K_1+K_2+\cdots+K_l, H)=1\\ (K_1K_2\cdots K_l, H)=1 \\K_1\equiv U \pmod{D}\\K_1+K_2+\cdots+K_l \equiv S \pmod M}} 1 \\
&= \sum_{\substack{M \mid H}}(\mu \ast F)(M)\sum_{\substack{B_1,B_2,\ldots, B_s \pmod{H}\\ M \mid B_1,\ldots, M \mid B_s}} \lambda_1(B_1) \cdots \lambda_s(B_s)\sum_{U \pmod{D}}\psi(U) \\
&  \times \sum_{\substack{ K_1, K_2,\ldots, K_l\pmod{H}\\ (K_1K_2\cdots K_l, H)=1\\K_1 \equiv U \pmod D \\K_1+K_2+\cdots+K_l \equiv S \pmod M}}\sum_{N \mid(K_1+K_2+\cdots+K_l, H)}\mu(N) \\
&=\sum_{\substack{M \mid H}}(\mu \ast F)(M)\sum_{\substack{B_1,B_2,\ldots, B_s \pmod{H}\\ M \mid B_1,\ldots, M \mid B_s}} \lambda_1(B_1) \cdots \lambda_s(B_s)\sum_{U \pmod{D}}\psi(U)\;\; \sum_{N\mid H} \mu(N) \sum_{\substack{ K_1, K_2,\ldots, K_l\pmod{H}\\ (K_1K_2\cdots K_l, H)=1\\K_1 \equiv U \pmod D \\K_1+K_2+\cdots+K_l \equiv S \pmod M\\K_1+K_2+\cdots+K_l \equiv 0 \pmod N}}1 \\
&= \sum_{\substack{M \mid H}}(\mu \ast F)(M)\sum_{\substack{B_1,B_2,\ldots, B_s \pmod{H}\\ M \mid B_1,\ldots, M \mid B_s}} \lambda_1(B_1) \cdots \lambda_s(B_s)\;\;\sum_{N \mid H} \mu(N)\sum_{U \pmod D}\psi(U)N_l(H, M, N, D,S,U), 
\end{split}  
\end{equation}
where $N_l(H, M, N, D,S,U)=\displaystyle \sum_{\substack{ K_1, K_2,\ldots, K_l\pmod{H}\\ (K_1K_2\cdots K_l, H)=1\\K_1 \equiv U \pmod D \\K_1+K_2+\cdots+K_l \equiv S \pmod M\\K_1+K_2+\cdots+K_l \equiv 0 \pmod N}}1$.

\medskip

We now evaluate the sum $N_l(H, M, N, D,S,U)$, where $M\mid H$, $N\mid H$ and $U \in \mathbb{A}$ such that $(U,D)=1$ be fixed. Since $(S, H)=1$, if $(M,N) >1$, then $N_l(H, M, N, D,S,U)=0$, the empty sum. Hence we assume that $(M, N) =1$.

If $l=1$, then from Lemma \ref{lem2}, we have 
\begin{equation}
\begin{split}
   N_1(H, M, N, D,S,U)& =\displaystyle \sum_{\substack{ K_1 \pmod{H}\\ (K_1, H)=1\\K_1 \equiv U \pmod D \\K_1\equiv S \pmod M\\K_1 \equiv 0 \pmod N}}1\\
    & = \begin{cases}
       	\frac{\varphi( H)}{\varphi([D,M])} & \text{ if }  N=1 , (U, D)=(S,M)=1 \text{ and } U \equiv S \pmod{(D,M)},\\ 
			0 & \text{ otherwise. }   
     \end{cases} \\
\end{split}
\end{equation}

\medskip

\begin{lemma}{ (Recursion formula) } \label{rec}
Let $H,M,N,U \in \mathbb{A}$ be such that $M, N \mid H$ and $(M,N) =1$. Then 
\begin{equation}
    N_l(H, M, N, D,S,U) = \frac{\varphi(H)}{\varphi(M) \varphi(N)} \displaystyle\sum_{J|M} \mu(J) \displaystyle\sum_{I|N} \mu (I) N_{l-1}(H, J, I, D ,S,U).
\end{equation}
\end{lemma}
\begin{proof}
We have
\begin{equation*}
\begin{split}
N_l(H, M, N, D,S,U) &=\displaystyle \sum_{\substack{ K_1, K_2,\ldots, K_l\pmod{H}\\ (K_1K_2\cdots K_l, H)=1\\K_1 \equiv U \pmod D \\K_1+K_2+\cdots+K_l \equiv S \pmod M\\K_1+K_2+\cdots+K_l \equiv 0 \pmod N}}1   \\
&= \displaystyle \sum_{\substack{ K_1, K_2,\ldots, K_{l-1} \pmod{H}\\ (K_1K_2\cdots K_{l-1}, H)=1\\K_1 \equiv U \pmod D  }} \displaystyle \sum_{\substack{ K_l \pmod{H}\\ (K_l, H)=1\\  K_l \equiv S-K_1-K_2-\cdots-K_{l-1} \pmod M\\  K_l \equiv -K_1-K_2-\cdots-K_{l-1} \pmod N}} 1  . \\
\end{split}
\end{equation*}
Since $(M,N)=1 $, Using Lemma \ref{lem2}, we get
\begin{equation*}
\begin{split}
N_l(H, M, N, D,S,U) &=\displaystyle \sum_{\substack{ K_1, K_2,\ldots, K_{l-1}\pmod{H}\\ (K_1K_2\cdots K_{l-1}, H)=1\\K_1 \equiv U \pmod D \\(K_1+K_2+\cdots+K_{l-1} -S,M) =1 \\(K_1+K_2+\cdots+K_{l-1},N) =1}} \frac{\varphi(H)}{\varphi(M)\varphi(N)}  \\
&= \frac{\varphi(H)}{\varphi(M)\varphi(N)} \displaystyle \sum_{\substack{ K_1, K_2,\ldots, K_{l-1}\pmod{H}\\ (K_1K_2\cdots K_{l-1}, H)=1\\K_1 \equiv U \pmod D }} \; \displaystyle \sum_{\substack{ J \mid (K_1+\cdots+K_{l-1} -S,M)}} \mu (J)\displaystyle \sum_{\substack{ I\mid (K_1+\cdots+K_{l-1},N)}} \mu(I) \\
&=  \frac{\varphi(H)}{\varphi(M)\varphi(N)} \displaystyle \sum_{J\mid M} \mu(J) \sum_{I\mid N} \mu(I) \; \sum_{\substack{ K_1, K_2,\ldots, K_{l-1}\pmod{H}\\ (K_1K_2\cdots K_{l-1}, H)=1\\K_1 \equiv U \pmod D\\ K_1+\cdots+K_{l-1} \equiv S \pmod J \\K_1+\cdots+K_{l-1} \equiv 0 \pmod I }} 1 \\
&= \frac{\varphi(H)}{\varphi(M)\varphi(N)} \displaystyle \sum_{J\mid M} \mu(J) \sum_{I\mid N} \mu(I)\; N_{l-1}(H, J, I, D,S,U).
\end{split}
\end{equation*}
\end{proof}

\begin{lemma} \label{lemma4.2}
Let $M,N \in \mathbb{A}$ such that $M \mid H$, $N \mid H$ with $(M,N)=1$. Then for every integer $l \geq 2$,
\begin{multline*}
	\sum_{\substack{ U \mod{ D} \\ (U,D)=1}} \psi (U) \; N_l(H, M, N, D,S,U) = \frac{\varphi(H)^l \psi(S) \mu(D)^{l-1}}{\varphi(M) \varphi(N)} \prod_{P\mid D} \frac{1}{(|P|-1)^{l-1}} \\
	  \prod_{\substack{P \mid M \\ P \nmid  D }} \left( 1-\frac{1}{(|P|-1)}+ \cdots+ (-1)^{l-1}\frac{1}{(|P|-1)^{l-1}} \right) \; \prod_{P \mid N } \left( 1-\frac{1}{(|P|-1)}+ \cdots+ (-1)^{l-2}\frac{1}{(|P|-1)^{l-2}} \right),
\end{multline*}
if $D \mid M$. Otherwise, the sum is 0.
\end{lemma}
\begin{proof}
We prove the lemma by induction on $l$. If $l=2,$ then 
\begin{equation*}
\begin{split}
    \displaystyle\sum_{\substack{ U \pmod{ D} \\ (U,D)=1}} \psi (U) N_2(H, M, N, D,S,U) &= \sum_{\substack{ U \pmod{ D} \\ (U,D)=1}}\psi (U) \frac{\varphi(H)}{\varphi(M)\varphi(N)} \displaystyle \sum_{J\mid M} \mu(J) \sum_{I\mid N} \mu(I) N_1(H, J, I, D,S,U) \\
    &= \frac{\varphi(H)}{\varphi(M)\varphi(N)} \sum_{\substack{ U \pmod{D} \\ U\equiv S((J,D))}} \psi (U) \sum_{J\mid M} \mu(J) \sum_{\substack{I\mid N \\ I=1}} \mu(I) \frac{\varphi(H)}{\varphi([J,D])}\\
    &= \frac{\varphi(H)^2}{\varphi(M)\varphi(N)}  \quad \sum_{J\mid M} \frac{\mu(J)}{\varphi([J,D])} \;\; \sum_{\substack{ U \pmod{D} \\ U\equiv S((J,D))}} \psi (U) .\\
\end{split}
\end{equation*}
From Lemma \ref{char}, we get $$\sum_{\substack{ U \pmod{D} \\ U\equiv S((J,D))}} \psi (U) \neq 0 \;\;\;\; \text{  if and only if } \;\;\;\; (J,D) =D \;\; \text{ that is, if and only if } \;\; D\mid J.$$ 
Furthermore, if $D\mid J$, then $$\sum_{\substack{ U \pmod{D} \\ U\equiv S((J,D))}} \psi (U) = \psi(S).$$ 
Hence 
\begin{align*}
\displaystyle\sum_{\substack{ U \pmod{ D} \\ (U,D)=1}} \psi (U) N_2(H, M, N, D,S,U) =  \frac{\varphi(H)^2 \psi(S)}{\varphi(M)\varphi(N)}   \sum_{\substack{ J\mid M \\ D\mid J}} \frac{\mu(J)}{ \varphi(D)}.
\end{align*}
Thus if $D\nmid M $ or if $D$ is not square-free, the sum is empty. If $D\mid M$, then 
\begin{equation*}
 \displaystyle\sum_{\substack{ U \mod{ D} \\ (U,D)=1}} \psi (U) N_2(H, M, N, D,S,U) = \frac{\varphi(H)^2 \psi(S) \mu(D)}{\varphi(M) \varphi(N)} \prod_{P\mid D} \frac{1}{(|P|-1)} \prod_{\substack{ P\mid M \\ P\nmid D}} \left( 1-\frac{1}{(|P|-1)} \right).
\end{equation*}
Hence the formula is true for $l=2.$  Assume that the formula is true for $l-1$. Then by the recursion Lemma \ref{rec}, we get
\begin{equation*}
\begin{split}
& \sum_{\substack{ U \pmod{ D} \\ (U,D)=1}} \psi (U) N_l(H, M, N, D,S,U)\\
&=\sum_{\substack{ U \pmod{ D} \\ (U,D)=1}}\psi (U) \frac{\varphi(H)}{\varphi(M)\varphi(N)} \sum_{J\mid M} \mu(J) \sum_{I\mid N} \mu(I) N_{l-1}(H, J, I, D,S,U) \\
&= \frac{\varphi(H)}{\varphi(M)\varphi(N)} \sum_{J\mid M} \mu(J) \sum_{ \substack{I\mid N \\ (I,J)=1 }} \mu(I) \sum_{\substack{ U \pmod{D} \\ (U,D)=1}} \psi (U) \; N_{l-1}(H, J, I, D,S,U) \\
 &= \frac{\varphi(H)}{\varphi(M)\varphi(N)} \sum_{J\mid M} \mu(J) \sum_{ \substack{I\mid N \\ (I,J)=1 }} \mu(I) \frac{\varphi(H)^{l-1} \psi(S) \mu(D)^{l-2}}{\varphi(J) \varphi(I)} \prod_{P\mid D} \frac{1}{(|P|-1)^{l-2}} \\
& \times   \prod_{\substack{P \mid J \\ P \nmid  D }} \left( 1-\frac{1}{(|P|-1)}+ \cdots+ \frac{(-1)^{l-2}}{(|P|-1)^{l-2}} \right)  \prod_{P \mid I } \left( 1-\frac{1}{(|P|-1)}+ \cdots+ \frac{(-1)^{l-3}}{(|P|-1)^{l-3}} \right)\\ 
&= \frac{\varphi(H)^l \psi(S) \mu(D)^{l-2}}{\varphi(M) \varphi(N)}  \prod_{P\mid D} \frac{1}{(|P|-1)^{l-2}} \;\; \sum_{\substack{J\mid M \\ D\mid J}} \frac{\mu(J)}{\varphi(J)}\;\; \prod_{\substack{P \mid J \\ P \nmid  D }} \left( 1-\frac{1}{(|P|-1)}+ \cdots+ \frac{(-1)^{l-2}}{(|P|-1)^{l-2}} \right)\\
& \times  \sum_{\substack{I \mid N }} \frac{\mu(I)}{\varphi(I)} \;\; \prod_{P \mid I} \left( 1-\frac{1}{(|P|-1)}+ \cdots+ \frac{(-1)^{l-3}}{(|P|-1)^{l-3}} \right) \\ 
&= \frac{\varphi(H)^l \psi(S) \mu(D)^{l-2}}{\varphi(M) \varphi(N)}\prod_{P\mid D} \frac{1}{(|P|-1)^{l-2}}\; \mu(D) \prod_{P\mid D}\frac{1}{(|P|-1)}  \\
& \times  \prod_{\substack{P \mid M \\ P \nmid  D }}  1- \frac{1}{|P|-1} \left( 1- \frac{1}{|P|-1}+ \cdots+ \frac{(-1)^{l-2}}{(|P|-1)^{l-2}} \right)  \\
& \times \prod_{\substack{P \mid N }}  1- \frac{1}{|P|-1} \left( 1- \frac{1}{|P|-1}+ \cdots+ \frac{(-1)^{l-3}}{(|P|-1)^{l-3}}\right) \\
&=  \frac{\varphi(H)^l \psi(S) \mu(D)^{l-1}}{\varphi(M) \varphi(N)} \prod_{P\mid D} \frac{1}{(|P|-1)^{l-1}}\;  \prod_{\substack{P \mid M \\ P \nmid  D }} \left( 1-\frac{1}{|P|-1}+ \cdots+ \frac{(-1)^{l-1}}{(|P|-1)^{l-1}} \right)  \\ 
&  \times  \prod_{P \mid N } \left( 1-\frac{1}{|P|-1}+ \cdots+ \frac{(-1)^{l-2}}{(|P|-1)^{l-2}} \right).
\end{split}
\end{equation*}

\end{proof}

\begin{lemma} \label{lemma4.3}
Let $H,B,M \in \mathbb{A}$. Let $\lambda $ be a additive character with $W \in \mathbb{A}$ corresponds to $\lambda.$ Then
$$
 \sum_{\substack{ B\pmod{H} \\ M\mid B}} \lambda(B) = \begin{cases}
       	|\frac{H}{M}| & \text{ if } \frac{H}{M} \mid W,\\ 
			0 & \text{ otherwise. }   
     \end{cases}
$$
\end{lemma}
\begin{proof}
Now
\begin{equation*}
\begin{split}
    \sum_{\substack{ B \pmod{H} \\ M\mid B}} \lambda(B) & = \sum_{C \pmod{H \mid M}} \lambda(MC) \\
    & = \sum_{C \pmod{H \mid M} } E( MC, H) (W) \\
    & = \sum_{C \pmod{H \mid M} } E(C, \frac{H}{M})(W) \\
    & = \begin{cases}
       	|\frac{H}{M}| & \text{ if } \frac{H}{M} \mid W,\\ 
			0 & \text{ otherwise. }   
     \end{cases}
\end{split}
\end{equation*}
\end{proof}

\noindent{\bf Proof of the Theorem \ref{MAIN-TH2} :} 

Using  Lemma \ref{lemma4.2} and Lemma \ref{lemma4.3}, equation \eqref{sum} becomes
\begin{equation*}
\begin{split}
\mathcal{T} & = \sum_{\substack{M \mid H \\ D\mid M \\\frac{H}{M} \mid W_i}}(\mu \ast F)(M) \left|\frac{H}{M}\right|^s \sum_{\substack{ N\mid H \\ (M,N ) =1 }}  \frac{\varphi(H)^l \psi(S) \mu(D)^{l-1}}{\varphi(M) \varphi(N)} \prod_{P\mid D} \frac{1}{(|P|-1)^{l-1}} \\
& \times  \displaystyle\prod_{\substack{P \mid M \\ P \nmid  D }} \left( 1-\frac{1}{(|P|-1)}+ \cdots+ \frac{(-1)^{l-1}}{(|P|-1)^{l-1}} \right)  \times  \displaystyle\prod_{P \mid N } \left( 1-\frac{1}{(|P|-1)}+\cdots+ \frac{(-1)^{l-2}}{(|P|-1)^{l-2}} \right) \\
& =\varphi(H)^l \psi(S) \mu(D)^{l-1}  \prod_{P\mid D} \frac{1}{(|P|-1)^{l-1}} \sum_{\substack{M \mid H \\ D\mid M \\ \frac{H}{M} \mid W_i}}\frac{ (\mu \ast F)(M)  }{ \varphi(M)} \left|\frac{H}{M}\right|^s \displaystyle\prod_{\substack{P \mid M \\ P \nmid  D }} \left( 1-\frac{1}{(|P|-1)}+\cdots+ \frac{(-1)^{l-1}}{(|P|-1)^{l-1}} \right) \\
& \times \sum_{ \substack{N\mid H\\ (M,N)=1  }} \frac{\mu(N)}{\varphi(N)} \;\; \displaystyle\prod_{P \mid N } \left( 1-\frac{1}{(|P|-1)}+\cdots+ \frac{(-1)^{l-2}}{(|P|-1)^{l-2}} \right) \\
&= \varphi(H)^l \psi(S) \mu(D)^{l-1}  \prod_{P\mid D} \frac{1}{(|P|-1)^{l-1}} \sum_{\substack{M \mid H \\ D\mid M \\ \frac{H}{M} \mid W_i}}\frac{ (\mu \ast F)(M)  }{ \varphi(M)} \left|\frac{H}{M}\right|^s   \\
&  \times \displaystyle\prod_{\substack{P \mid M }} \left( 1-\frac{1}{(|P|-1)}+ \cdots+ \frac{(-1)^{l-1}}{(|P|-1)^{l-1}} \right)  \displaystyle\prod_{\substack{P \mid D }} \left( 1-\frac{1}{(|P|-1)}+ \cdots+ \frac{(-1)^{l-1}}{(|P|-1)^{l-1}} \right)^{-1}   \\ 
& \times \displaystyle\prod_{\substack{P \mid H \\ P \nmid M }} 1-\frac{1}{|P|-1} \left( 1-\frac{1}{(|P|-1)}+ \cdots+ \frac{(-1)^{l-2}}{(|P|-1)^{l-2}} \right)   \\
\end{split}
\end{equation*}

\begin{equation*}
\begin{split}
 & = \varphi(H)^l \psi(S) \mu(D)^{l-1}  \prod_{P\mid D} \frac{1}{(|P|-1)^{l-1}} \sum_{\substack{M \mid H \\ D\mid M \\ \frac{H}{M} \mid W_i}}\frac{ (\mu \ast F)(M)  }{ \varphi(M)} \left|\frac{H}{M}\right|^s   \\
& \quad \quad  \times \displaystyle\prod_{\substack{P \mid M }} \left( 1-\frac{1}{(|P|-1)}+ \cdots+ \frac{(-1)^{l-1}}{(|P|-1)^{l-1}} \right)  \displaystyle\prod_{\substack{P \mid D }} \left( 1-\frac{1}{(|P|-1)}+ \cdots+ \frac{(-1)^{l-1}}{(|P|-1)^{l-1}} \right)^{-1}   \\ 
&  \quad \quad   \times \displaystyle\prod_{\substack{P \mid H}} \left( 1-\frac{1}{(|P|-1)}+ \cdots+ \frac{(-1)^{l-1}}{(|P|-1)^{l-1}} \right)\displaystyle\prod_{\substack{P \mid M}} \left( 1-\frac{1}{(|P|-1)}+ \cdots+ \frac{(-1)^{l-1}}{(|P|-1)^{l-1}} \right)^{-1}.
\end{split}
\end{equation*}
Since $$\varphi_{l}(H)=\varphi(H)^l \prod_{P \mid H} \left(1-\frac{1}{|P|-1}+\frac{1}{(|P|-1)^2}+\cdots + \frac{(-1)^{l-1}}{(|P|-1)^{l-1}}\right),$$ we have

\begin{equation*}
\begin{split}	
\mathcal{T}& =  \varphi_{l}(H) \psi(S) \mu(D)^{l-1} \prod_{P\mid D} \frac{1}{(|P|-1)^{l-1}} \displaystyle\prod_{\substack{P \mid D }} \left( 1-\frac{1}{(|P|-1)}+ \cdots+ \frac{(-1)^{l-1}}{(|P|-1)^{l-1}} \right)^{-1}\\
& \times \sum_{\substack{M \mid H \\ D\mid M \\ \frac{H}{M} \mid W_i}}\frac{ (\mu \ast F)(M)  }{ \varphi(M)} \left|\frac{H}{M}\right|^s\\
&= \mu(D)^{l-1} \; \psi(S) \; \varphi\left( \frac{H_0^l}{D^{l-1}}\right) \varphi_l \left( \frac{H}{H_0}\right) \sum_{\substack{D\mid M \mid H \\ \frac{H}{M} \mid W_i}}\frac{ (\mu \ast F)(M)  }{ \varphi(M)} \left|\frac{H}{M}\right|^s,
\end{split}
\end{equation*}
where $H_0 \mid H$ is such that $H_0$ has the same prime divisors as that of $D$ and $\left(H_0, \frac{H}{H_0} \right)=1$. This completes the proof.
\qed

{\bf Acknowledgements.}
The first named author thanks the Ramakrishna Mission Vivekananda Educational and Research Institute for an excellent work environment and thanks UGC NET fellowship for providing financial support.

\end{document}